\documentclass[a4paper,12pt]{amsart}

\usepackage[english]{babel}
\usepackage{amsthm, latexsym, gastex, amssymb}
\usepackage[mathcal]{eucal}
\usepackage{graphicx}

\DeclareMathOperator{\Imm}{Im}
\DeclareMathOperator{\lcm}{lcm}

\renewcommand{\phi}[0]{\varphi}
\renewcommand{\theta}[0]{\vartheta}
\renewcommand{\epsilon}[0]{\varepsilon}

\newcommand{\Z}{\text{$\mathbf{Z}$}}
\newcommand{\Q}{\text{$\mathbf{Q}$}}

\newcommand{\Pro}{\text{$\mathbf{P}^1$}}
\newcommand{\F}{\text{$\mathbf{F}$}}

\newcommand{\K}{\text{$\mathbf{K}$}}

\newcommand{\Fq}{\text{$\mathbf{F}_q$}}

\newcommand{\Kp}{\text{$\mathit{{Kob}}_0$}}
\newcommand{\Kob}{\text{$\mathit{Kob}_0$}}

\newtheorem{theorem}{Theorem}[section]
\newtheorem{lemma}[theorem]{Lemma}

\theoremstyle{definition}

\newtheorem{example}[theorem]{Example}

\theoremstyle{remark}
\newtheorem{remark}[theorem]{Remark}

\numberwithin{equation}{section}

\DeclareMathOperator{\Tr}{Tr}
\DeclareMathOperator{\End}{End}

\begin{document}

\bibliographystyle{amsalpha}

\title[Graphs associated with the map $X \mapsto X + X^{-1}$]{Graphs associated with the map $X \mapsto X + X^{-1}$\\
in finite fields of characteristic two\\}

\author{S.~Ugolini}
\email{sugolini@gmail.com}

\date{August 30, 2011}

\begin{abstract}
In this paper we study the structure of the graphs associated with the iterations of the map $x \mapsto x+x^{-1}$ over finite fields of characteristic two. Formulas are given for the length of the cycles and the depth of the trees relying upon the structure of the group of the rational points of Koblitz curves and the congruences of Kloosterman sums modulo powers of 2.
\end{abstract}

\maketitle

\section{Introduction}
The map which sends $x$ to $x + x^{-1}$ in a finite field (with a point $\infty$ added to it) plays a role in various investigations. The so-called $Q$-transform depends on it, as it takes a polynomial $f$ of degree $n$ to the self-reciprocal polynomial $f^Q(x) = x^n f(x + x^{-1})$ of degree $2n$ (see \cite{jung}). Also, the possible correlation between the multiplicative orders of $x$ and $x +x^{-1}$ was studied in \cite{shpa}.

Iteration of maps on finite fields are also important. For example, Pollard's integer factoring algorithm is based on the iteration of a quadratic map $x \mapsto x^2 + c \pmod{N}$, where $c \not = 0,-2$ is a randomly-chosen constant and $N$ is the integer to be factored. See \cite{vasiga} for one of several studies on iterations of maps of this form in a finite field.

Our work focuses on iterations of the map $x \mapsto x + x^{-1}$ on the projective line $\Pro (\F_{2^n})= \F_{2^n} \cup \{\infty\}$, where $\F_{2^n}$ is a finite field of characteristic 2. A directed graph on $\Pro (\F_{2^n})$ is associated with the map in an obvious way. Each connected component consists of a cycle and directed binary trees entering the cycle at various points.

Experimental evidence has shown that such graphs present remarkable symmetries. In fact, it turns out that the map is closely related to the duplication map on a certain elliptic curve on $E$, the Koblitz curve $y^2 + xy = x^3 + 1$ over $\F_2$. Using this fact we give a precise description of the structure of such graphs, including the length of the cycles and the depth of the trees.

\section{Preliminaries}
\label{preliminaries}
For a fixed positive integer $n$ let $\F_{2^n}$ be the field with $2^n$ elements and $\Pro ( \F_{2^n}) = \F_{2^n} \cup \left\{ \infty \right\}$ the projective line over $\F_{2^n}$. We define a map $\theta$ over $\Pro ( \F_{2^n})$ in such a way:
\begin{displaymath}
\theta (\alpha) = \left\{
\begin{array}{lll}
\infty & \text{if $\alpha = 0$ or $\infty$}\\
\alpha + \alpha^{-1} & \textrm{otherwise}.
\end{array}
\right.
\end{displaymath}
We can associate a graph with the map $\theta$ over the field $\F_{2^n}$ in a natural way. The vertices of the graph are labelled by the elements of $\Pro (\F_{2^n})$. If $\alpha \in \Pro (\F_{2^n})$ and $\beta = \theta(\alpha)$, then we connect with a directed edge the vertex $\alpha$ with the vertex $\beta$.
If $\gamma \in \Pro (\F_{2^n})$ and $\theta^k (\gamma) = \gamma$, for some positive integer $k$, then $\gamma$ belongs to a cycle of length $k$ or a divisor of $k$.
An element $\gamma$ belonging to a cycle can be the root of a reverse-directed tree, provided that $\gamma = \theta (\alpha)$, for some $\alpha$ which is not contained in any cycle. 

\begin{example}
Consider the graph associated with the map $\theta$ in the field $\F_{2^5}$, constructed as the splitting field  over $\F_2$ of the polynomial $x^5+x^2+1 \in \F_2[x]$. If $\alpha$ is a root of such a polynomial in $\F_{2^5}$, then $\Pro (\F_{2^5}) = \{ 0 \} \cup \{ \alpha^i : 1 \leq i \leq 31 \} \cup \{ \infty \}$. We will label the nodes denoting the elements $\alpha^i$ by the exponent $i$ and the zero element by `0'.

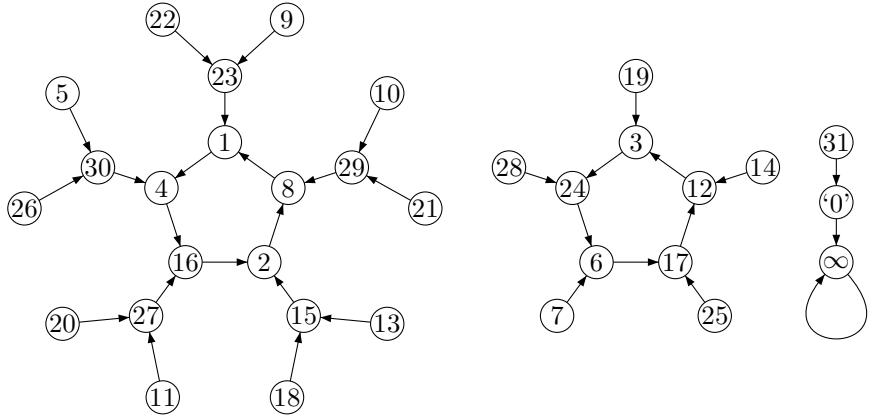
\begin{center}
    \unitlength=2.5pt
    \begin{picture}(70, 80)(-35,-40)
    \gasset{Nw=5,Nh=5,Nmr=2.5,curvedepth=0}
    \thinlines
    \footnotesize
    \node(N1)(0,10){$1$}
    \node(N2)(9.5,3.1){$8$}
    \node(N3)(5.9,-8.1){$2$}
    \node(N4)(-5.9,-8.1){$16$}
    \node(N5)(-9.5,3.1){$4$}
    \node(N6)(0,20){$23$}
    \node(N7)(19,6.2){$29$}
    \node(N8)(11.8,-16.2){$15$}
    \node(N9)(-11.8,-16.2){$27$}
    \node(N10)(-19,6.2){$30$}
    \node(N11)(9.3,28.5){$9$}
    \node(N12)(24.3,17.4){$10$}
    \node(N13)(30,0){$21$}
    \node(N14)(24.3,-17.4){$13$}
    \node(N15)(9.3,-28.5){$18$}
    \node(N16)(-9.3,-28.5){$11$}
    \node(N17)(-24.3,-17.4){$20$}
    \node(N18)(-30,0){$26$}
    \node(N19)(-24.3,17.4){$5$}
    \node(N20)(-9.3,28.5){$22$}

    \drawedge(N2,N1){}
    \drawedge(N3,N2){}
    \drawedge(N4,N3){}
    \drawedge(N5,N4){}
    \drawedge(N1,N5){}
    \drawedge(N6,N1){}
    \drawedge(N7,N2){}
    \drawedge(N8,N3){}
    \drawedge(N9,N4){}
    \drawedge(N10,N5){}
    \drawedge(N11,N6){}
    \drawedge(N12,N7){}
    \drawedge(N13,N7){}
    \drawedge(N14,N8){}
    \drawedge(N15,N8){}
    \drawedge(N16,N9){}
    \drawedge(N17,N9){}
    \drawedge(N18,N10){}
    \drawedge(N19,N10){}
    \drawedge(N20,N6){}
\end{picture}
\begin{picture}(60, 65)(-25,-40)
    \gasset{Nw=5,Nh=5,Nmr=2.5,curvedepth=0}
    \unitlength=2.5pt
    \thinlines
      \footnotesize
    \node(N1)(0,10){$3$}
    \node(N2)(9.5,3.1){$12$}
    \node(N3)(5.9,-8.1){$17$}
    \node(N4)(-5.9,-8.1){$6$}
    \node(N5)(-9.5,3.1){$24$}
    \node(N6)(0,20){$19$}
    \node(N7)(19,6.2){$14$}
    \node(N8)(11.8,-16.2){$25$}
    \node(N9)(-11.8,-16.2){$7$}
    \node(N10)(-19,6.2){$28$}
    \drawedge(N2,N1){}
    \drawedge(N3,N2){}
    \drawedge(N4,N3){}
    \drawedge(N5,N4){}
    \drawedge(N1,N5){}
    \drawedge(N6,N1){}
    \drawedge(N7,N2){}
    \drawedge(N8,N3){}
    \drawedge(N9,N4){}
    \drawedge(N10,N5){}
    \node(N11)(30,10){$31$}
    \node(N12)(30,0.95){`0'}
    \node(N13)(30,-8.1){$\infty$}
    \drawloop[loopangle=-90](N13){}
    \drawedge(N11,N12){}
    \drawedge(N12,N13){}
    \end{picture}
\end{center}
\end{example}
In the following we will denote the absolute trace of an element $\alpha \in \F_{2^n}$ by $\Tr_n (\alpha)$, namely
\begin{displaymath}
\Tr_{n} (\alpha) = \sum_{i=0}^{n-1} {\alpha}^{2^i}.
\end{displaymath}
Since $\Tr_{n} (\alpha) \in \F_2$, the set of points belonging to the projective line $\Pro (\F_{2^n})$ can be partitioned in the subsets
\begin{eqnarray*}
A_n &  = & \{ \alpha \in \F_{2^n}^* : \Tr_{n} (\alpha) = \Tr_{n} (\alpha^{-1}) \} \cup \{ 0, \infty \} \\
B_n & = & \{ \alpha \in \F_{2^n}^* : \Tr_{n} ( \alpha ) \not = \Tr_{n} (\alpha^{-1})  \}.
\end{eqnarray*}

The following holds.
\begin{lemma}\label{tr_inverse}
If $\alpha$ is an element of $\F_{2^n}^*$, then
\begin{equation*}
\Tr_{n} \left((\alpha + \alpha^{-1})^{-1} \right) = 0.
\end{equation*}
\end{lemma}
\begin{proof}
We compute explicitly the trace of $\beta = (\alpha + \alpha^{-1})^{-1}$:
\begin{eqnarray*}
\Tr_{n} (\beta) & = & \sum_{i=0}^{n-1} \left(\frac{\alpha}{\alpha^2 + 1} \right)^{2^i} = \sum_{i=1}^n \left(\frac{1}{\alpha^{2^{i-1}}+1} + \frac{1}{\alpha^{2^i}+1} \right) = 0.
\end{eqnarray*}
\end{proof}
\begin{remark}
As a consequence of previous Lemma, if one considers the restrictions $\theta_{A_n}$ and $\theta_{B_n}$ of $\theta$ at $A_n$ (respectively $B_n$), then $\Imm(\theta_{A_n}) \subseteq A_n$ and $\Imm(\theta_{B_n}) \subseteq B_n$. This amounts to saying that the graph associated with the map $\theta$ in the field $\F_{2^n}$ is the union of the graphs associated with the maps $\theta_{A_n}$ over $A_n$ and $\theta_{B_n}$ on $B_n$.  
\end{remark}

The map $\theta$ is strictly related to the duplication map defined over Koblitz curves. We remind that a Koblitz curve is an elliptic curve defined over $\F_2$ by an equation of the form
\begin{equation*}
y^2+xy = x^3+a x^2 +1,
\end{equation*}
where  $a \in \F_2$. In particular, for $a=0$ we get the  curve $\Kob$ defined by 
\begin{equation*}
y^2+xy=x^3+1.
\end{equation*} 
If $P = (x_1, y_1) \in \Kob (\F_{2^n})$, namely $P$ is a rational point of $\Kob$ over the field $\F_{2^n}$, then $2 P = (x_2, y_2)$, where
\begin{eqnarray*}
x_2 & = & x_1^2 + \frac{1}{x_1^2} = \theta(x_1)^2
\end{eqnarray*}
Moreover, if $P=(x,y) \in \Kob (\F_{2^n})$, then $-P=(x,x+y)$.

The following result holds (see \cite{LW}).
\begin{lemma}
Let $\beta \in \F_{2^n}$. Then, $\Tr_{n} (\beta) = 0$ if and only if there exists $\alpha \in \F_{2^n}$ such that $\beta = \alpha + {\alpha}^{2}$.
\end{lemma}

We make immediately use of the Lemma above, proving the following.

\begin{lemma}\label{pre_1}
Let $x \in \F_{2^n}$. Then, there exists $y \in \F_{2^n}$ such that $(x,y) \in \Kob (\F_{2^n})$ if and only if $x = 0$ or $\Tr_{n}(x) = \Tr_{n}(x^{-1})$.
\end{lemma}
\begin{proof}
Let $y \in \F_{2^n}$ such that $(x,y) \in \Kob (\F_{2^n})$. If $x \not = 0$, then
\begin{equation*}
\frac{y^2}{x^2}+\frac{y}{x} = x + x^{-2}.
\end{equation*}
Since $\Tr_{n} \left(\dfrac{y^2}{x^2} \right) = \Tr_{n} \left(\dfrac{y}{x} \right)$ and $\Tr_{n} (x^{-1}) = \Tr_{n} (x^{-2})$, then $\Tr_{n} (x+ x^{-1}) = 0$.

Conversely, if $x= 0$, then $(x,y)=(0,1) \in \Kob (\F_{2^n})$. If $x \not = 0$ and $\Tr_{n} (x) = \Tr_{n} (x^{-1})$, then the equation
\begin{equation*}
z^2 + z = x + x^{-2}
\end{equation*}
has two solutions $z_1, z_2$. Let $y_i = z_i \cdot x$, for $i = 1$ or $2$. Since
\begin{equation*}
\frac{{y_i}^2}{x^2}+\frac{y_i}{x} = x+x^{-2},
\end{equation*}  
we get that ${y_i}^2 + x y_i = x^3 +1$ and we are done.
\end{proof}

If $a$ is an element of $\F_{2^n}$, then we can define the {Kloosterman sum}
\begin{displaymath}
S^{(n)}(a) = \sum_{x \in \F_{2^n}^*} (-1)^{\Tr_{n} (x^{-1}+ax)}.
\end{displaymath}

The values of the Kloosterman sums for $a=1$ are strictly related to the number of rational points of $\Kob$ over $\F_{2^n}$ (see \cite{LW} for more details). We have that 
\begin{equation*}
\lvert \Kob (\F_{2^n}) \rvert = 2^n+1+S^{(n)}(1).
\end{equation*}
In  \cite{CZ} some relations for $S^{(n)}(1)$ are given. It is proved that

\begin{equation}\label{K(1)mod8}
S^{(n)}(1) \equiv \left\{
\begin{array}{lll}
-1  & \pmod{8} & \textrm{for $n$ even, $n \not =2$}\\
3 & \pmod{8} & \textrm{for $n$ odd}. 
\end{array}
\right.
\end{equation}

Moreover
\begin{equation}\label{sn1}
S^{(n)} (1) = - 2 \cdot 2^{n/2} \cos(n \phi),
\end{equation}
where 

\begin{displaymath}
cos(\phi) = \frac{1}{-2 \sqrt{2}}; \quad \sin(\phi) = \frac{\sqrt{7}}{2 \sqrt{2}}.
\end{displaymath}

As a consequence of (\ref{sn1}) the following relation between $S^{(2n)}(1)$ and $S^{(n)}(1)$ holds:
\begin{equation}\label{s2nsn}
S^{(2n)} (1) = -S^{(n)}(1)^2 + 2^{n+1}.
\end{equation} 
Using  (\ref{s2nsn}) it is possible to generalize (\ref{K(1)mod8}).

\begin{lemma}\label{Sncon}
Let $n = 2^l m$ be a positive integer greater than 1 such that $n \not \in \{ 2, 4\}$, for some non-negative integer $l$ and odd integer $m$. 
Then 
\begin{displaymath}
S^{(n)} (1) \equiv \left\{
\begin{array}{ll}
-1 & \pmod {2^{l+2}} \\
-1+2^{l+2} &  \pmod {2^{l+3}}.
\end{array}
\right.
\end{displaymath}
\end{lemma} 

\begin{proof}
Let $m \geq 3$. We prove the thesis by induction on $l \geq 0$.

Let $l =0$. In this case $n$ is odd and the thesis follows  from (\ref{K(1)mod8}).

Suppose that the thesis holds for some non-negative integer $l-1$. Let $n=2^{l} m = 2(2^{l-1} m)$. For the sake of clarity denote $k=2^{l-1} m$. Then,
\begin{displaymath}
S^{(n)} (1) = - S^{(k)}(1)^2+2^{k+1} \equiv -1 \pmod{2^{l+2}},
\end{displaymath}
being $k+1=2^{l-1}m+1 \geq 2^{l-1} 3+1 \geq l+3$.
As regards the second congruence, since $S^{(k)} (1) \equiv -1 + 2^{l+1} \pmod{2^{l+2}}$, then $S^{(k)} (1)^2 \equiv 1 - 2^{l+2} \pmod{2^{l+3}}$. Hence,
\begin{displaymath}
S^{(n)} (1) \equiv -1+2^{l+2}+2^{k+1} \equiv -1+2^{l+2} \pmod{2^{l+3}}. 
\end{displaymath}

Now we deal with the case $n=2^l$, where $l \geq 3$.

When $n=8$, namely $l=3$, the Kloosterman sum $S^{(8)}(1) = 31$. Hence we are done.
Suppose that the thesis holds for some integer $l-1$ greater than or equal to 3. Let $n=2k$, where $k=2^{l-1}$. Then,
\begin{displaymath}
S^{(2k)} (1) = -S^{(k)}(1)^2+2^{k+1} \equiv -1 \pmod{2^{l+2}},
\end{displaymath}
since $k+1=2^{l-1}+1 \geq l+3$ for $l \geq 4$.

The second congruence holds too, since
\begin{displaymath}
S^{(2k)} (1) = -S^{(k)}(1)^2+2^{k+1} \equiv -1+2^{l+2} \pmod{2^{l+3}}.
\end{displaymath}
\end{proof}

\section{The structure of the  group of the rational points of an elliptic curve over a finite field}
In this Section  we will briefly recall some results concerning the number of rational points and the structure of the group of rational points of an elliptic curve defined over a finite field of arbitrary characteristic. More details can be found for example in \cite{Ruck} or \cite{wi}.

If $E$ is an elliptic curve defined over a finite field $\F_q$ of characteristic $p$, then the structure of the group $E(\F_{q^k})$ of the rational points of $E$ over $\F_{q^k}$ is strictly related to the ring $\End_{\Fq}(E)$ of the endomorphisms of $E$ over $\F_q$. 
Among all the endomorphisms, the Frobenius endomorphism $\pi_q$ plays a special role, as we will see later. It maps a point $P = (x_P,y_P)$ of $E$ to $(x_P^q, y_P^q)$. Note in passing that the ring $\Z$ of integers can be viewed as a subring of $\End_{\F_q} (E)$.

The following Theorem holds.
\begin{theorem}\label{num_poi}
Let $E$ be an elliptic curve defined over a finite field $\F_q$ and $h$ the number of rational points of $E$ over $\F_q$. Then,
\begin{equation*}
h = 1+q-\beta,
\end{equation*}
where $\beta$ is an integer with $\lvert \beta \rvert \leq 2 \sqrt{q}$. Moreover, if $(\beta, p)=1$, then $\Q (\pi_q)$ is an imaginary quadratic field over $\Q$ and all the orders in $\Q (\pi_q)$ are possible endomorphism rings of $E$ over $\F_q$. 
\end{theorem}

We remind that an order $\mathcal{D}$ in a number field $\K$ is a subring of $\K$ such that $\K$ is its quotient field, $\mathcal{D} \cap \Q = \Z$ and the additive group of $\mathcal{D}$ is finitely generated.

The structure of the group  $E(\F_{q^k})$ of rational points over $\F_{q^k}$ of an elliptic curve defined over a finite field $\F_q$ such that $\pi_q \not \in \Z$ is as follows. 
\begin{theorem}\label{iso_ell_comp}
Let $E$ be an elliptic curve defined over $\F_q$, $R = \End_{\Fq} (E)$ and $\F_{q^k}$ the field with ${q^k}$ elements. If $\pi_q \not \in \Z$, then there is an isomorphism
\begin{displaymath}
E(\F_{q^k}) \cong R/(\pi_q^k -1) R
\end{displaymath}
of $R$-modules.
\end{theorem}

\begin{theorem}\label{Frob_rep_theo}
In the same hypotheses of Theorem \ref{iso_ell_comp}, if $m = | E(\Fq) |$ and $d = (q+1-m)^2-4q$, then $d < 0$ and $\End_{\Fq} (E)$ is an order in $\Q (\sqrt{d})$.
\end{theorem}

The proof of Theorem \ref{Frob_rep_theo}, which can be found in \cite{wi}, yields a representation of the $q$-Frobenius endomorphism as an element of $\Q (\sqrt{d})$, namely
\begin{equation}\label{Frob_rep}
\pi_q = \frac{q+1-m+\sqrt{d}}{2}.
\end{equation}

Consider now the Koblitz curve $\Kob$. The number of rational points of $\Kob$ over $\F_2$ is $4$. Hence the representation of the Frobenius endomorphism $\pi_2$ as an element of $\Q (\sqrt{-7})$ is
\begin{equation*}
\pi_2  =  \frac{-1+i \sqrt{7}}{2}.
\end{equation*} 

We remind that the ring of integers of $\Q (\sqrt{-7})$, which is also its unique maximal order, is $\Z [\omega]$, where 
\begin{displaymath}
\omega = \frac{1+i \sqrt{7}}{2}.
\end{displaymath}
In particular we can write $\pi_2 = -1+\omega$.
Since $\End (\Kp)$ is an order in $\Q (\sqrt{-7})$ and contains $\{ 1, \pi_2 \}$, then $\omega \in \End(\Kob)$. Therefore $R=\End(\Kob)= \Z [\omega]$ and the group of rational points of $\Kob$ over $\F_{2^n}$ is isomorphic to $R/(\pi_2^n -1) R$.   
The ring $R$ is euclidean with respect to the norm $N(a+b \omega) = (a+b \omega) (\overline{a+b \omega})$. In particular $N(\pi_2) = 2$, namely
\begin{equation}
\pi_2 \overline{\pi}_2 = 2 \in R,
\end{equation}  
where $2$ is the duplication map, seen as an endomorphism of $\Kp$.

If $P = (x_1,y_1) \in \Kp (\F_{2^n})$ and $2 P = (x_2, y_2)$, then  $x_2=x_1^2+(x_1^{-1})^2$. Therefore, if $\overline{\pi}_2 (P) =(x', y')$, then $x' = x_1 + x_1^{-1}$.

We have obtained that the conjugated of the Frobenius endomorphism takes the $x$-coordinate of a point $P \in \Kob$ to $\theta(x)$. Relying upon this consideration we can study the structure of the graph associated with the map $\theta$ over a finite field of characteristic two.

\section{The structure of the graphs}
In this Section we will describe the structure of the graphs associated with the map $\theta$. We define the orbit of an element $x \in \Pro (\F_{2^n})$, under the action of the map $\theta$, as the set
\begin{displaymath}
\mathcal{O} (x) = \left\{ \theta^k (x) : k \geq 0 \right\}.
\end{displaymath}
The point $x$ is said to be periodic if $\theta^k (x) = x$, for some positive integer $k$. The smallest such $k$ is called the period of $x$.

The following holds.
\begin{lemma}\label{orb_1}
Let $x \in \F_{2^{2n}}$ and $P = (x, y) \in \Kp (\F_{2^{2n}})$. Denote by $O$ the point at infinity of $\Kob$. Then,
\begin{displaymath}
(\pi_2^n + 1) P = O \Longleftrightarrow \left\{
\begin{array}{l}
x = 0 \\
\text{\emph{or}} \\
x \in \F_ {2^n} \textrm{ and } \Tr_{n}  (x+x^{-1}) = 1. 
\end{array}
\right.
\end{displaymath}
\end{lemma} 

\begin{proof}
Suppose that $(\pi_2^n+1) P = O$. Then, $\pi_2^n (P) = -P = (x, x+y)$ and $x^{2^n} = x$. Hence, $x \in \F_{2^n}$. If $x \not = 0$ and $\Tr_{n} (x + x^{-1}) = 0$, then $y \in \F_{2^n}$, as a consequence of Lemma \ref{pre_1}. But this implies that $(x, y) \in \F_{2^n}^2$. Therefore $\pi_2^n (P) = P$ and $(x,y) = P = -P = (x, x+y)$. Hence $x=0$ and $y = 1$, in contradiction with the assumption that $x \not = 0$. 

Conversely, suppose that $x \in \F_{2^n}^*$ and $\Tr_{n} (x + x^{-1}) =1$. This implies that $\pi_2^n (P) = (x, y^{2^n})$. If $y^{2^n} = y$, then $P \in \Kp (\F_{2^n})$. But this is in contradiction with Lemma \ref{pre_1}. Hence, $(\pi_2^n+1) P = O$.
\end{proof}

Our final goal is to describe the structure of the graph associated with the map $\theta$ over  $\Pro(\F_{2^n})$. As remarked in Section \ref{preliminaries}, this can be done describing separately the structure of the graphs associated with the maps $\theta_{A_n}$ and $\theta_{B_n}$ on the sets $A_n$ and $B_n$ respectively (see Section \ref{preliminaries} for the details). 
\begin{itemize}
\item \emph{Graph $A_n$.} 
We remind that, for each $x \in A_n \backslash \{ 0, \infty \}$, there exist two distinct points in $\Kob(\F_{2^n})$ having the same $x$-coordinate. Moreover $\Kp (\F_{2^n})$ is isomorphic to $R / (\pi_2^n-1) R$, being $R$ the ring of integers of $\Q (\sqrt{-7})$. 
\item \emph{Graph $B_n$.}
Let $x \in B_n$. In this case there are exactly two distinct points in $\Kp(\F_{2^{2n}})$ with such an $x$-coordinate. By Lemma \ref{orb_1}, $(\pi_2^n + 1) P = O$. Viceversa, if $P= (x,y) \in \Kp(\F_{2^{2n}})$ and $(\pi_2^n+1) P = O$, then $x \in \F_{2^n}$ and $\Tr_{n} (x+x^{-1}) =1$ or $P=(0,1)$. Hence, there is an isomorphism 
\begin{equation*}
\widetilde{\psi} : \Kp(\F_{2^{2n}})_{B_n} \to R / (\pi_2^n+1) R,
\end{equation*}
where 
\begin{equation*}
\Kp(\F_{2^{2n}})_{B_n} = \{ (x,y) \in \Kp(\F_{2^{2n}}) : x \in B_n \text{ or } x = 0 \} \cup O.
\end{equation*}
\end{itemize} 

Before dealing with graphs $A_n$ and $B_n$ we recall just some facts about the ring $\Z[\omega]$, which is the ring of integers of the quadratic number field $\Q(\sqrt{-7})$. Since the ring $\Z[\omega]$ is euclidean, it is also a unique factorization domain. Moreover, the only positive rational prime which ramifies in $R$ is $7$, while all other (positive) rational primes either split in $R$ or are inert. 

We can factor the element $\pi_2^n-1$ (resp. $\pi_2^n+1$) in primes of $R$. Notice that $\pi_2$ divides neither $\pi_2^n-1$ nor $\pi_2^n+1$. 

Suppose that  $\pi_2^n-1$ (resp. $\pi_2^n+1$) factors as
\begin{equation*}
\overline{\pi}_2^{e_0} \cdot \left( \prod_{i=1}^v p_i^{e_i} \right) \cdot \left( \prod_{i = v+1}^w r_i^{e_i} \right) \cdot (\sqrt{-7})^f,
\end{equation*}
where 
\begin{enumerate}
\item all $e_i$ and $f$ are non-negative integers;
\item  for $1 \leq i \leq v$ the elements $p_i \in \Z$ are distinct primes of $R$ and $N( p_i^{e_i} ) = p_i^{2 e_i}$;
\item for $v+1 \leq i \leq w$ the elements $r_i \in R \backslash \Z$  are distinct primes of $R$, different from $\pi_2, \overline{\pi}_2$ and $\sqrt{-7}$,  and $N( r_i^{e_i} ) = p_i^{e_i}$, for some rational integer $p_i$ such that $r_i \overline{r}_i = p_i$.
\end{enumerate}

The ring $R / (\pi_2^n - 1) R$ (resp. $R / (\pi_2^n + 1) R$) is isomorphic to 
\begin{equation}\label{orb4}
R / \overline{\pi}_2^{e_0} R \times \left( \prod_{i=1}^v R / p_i^{e_i} R \right) \times \left( \prod_{i = v+1}^w R / r_i^{e_i} R \right) \times R /(\sqrt{-7})^f R.
\end{equation}

For any $1 \leq i \leq v$ the additive group of $R / p_i^{e_i} R$ is isomorphic to the direct sum of two cyclic groups of order $p_i^{e_i}$. This implies that, for each integer $0 \leq h_i \leq e_i$, there are $N_{h_i}$ points in $R / p_i^{e_i} R$ of order $p_i^{h_i}$, where
\begin{displaymath}
N_{h_i} = \left\{
\begin{array}{ll}
1 & \text{ if $h_i = 0$}\\
{p_i}^{2 h_i} - {p_i}^{2(h_{i}-1)} & \text{ otherwise.}
\end{array}
\right.
\end{displaymath}

For any $v+1 \leq i \leq w$ the additive group of $R / r_i^{e_i} R$ is cyclic of order $p_i^{e_i}$. Hence, there are $\phi(p_i^{h_i})$ points in $R / r_i^{e_i} R$ of order $p_i^{h_i}$, for each integer $0 \leq h_i \leq e_i$.

Finally, the additive group of $R / (\sqrt{-7})^f R$ is isomorphic to the direct sum of two cyclic groups of order $7^{f/2}$, if $f$ is even, or to the direct sum of two cyclic groups of order respectively $7^{(f-1)/2}$ and $7^{(f+1)/2}$, if $f$ is odd. In the case that $f$ is even, for each integer $0 \leq h_f \leq f/2$ there are $N_{h_f}$ points in $R/({\sqrt{-7}})^{f} R$ of order $7^{h_f}$, where  
\begin{displaymath}
N_{h_f} = \left\{
\begin{array}{ll}
1 & \text{ if $h_f = 0$}\\
7^{2 h_f} - 7^{2(h_{f}-1)} & \text{ otherwise.}
\end{array}
\right.
\end{displaymath}
If, on the contrary, $f$ is odd, then  
\begin{displaymath}
N_{h_f} = \left\{
\begin{array}{ll}
1 & \text{ if $h_f = 0$}\\
7^{2 h_f} - 7^{2(h_{f}-1)} & \text{ if $1 \leq h_f \leq (f-1)/2$}\\
7^{2 h_f-1} - 7^{2(h_f-1)} & \text{ if $h_f = (f+1)/2$}
\end{array}
\right.
\end{displaymath}

An element $x \in A_n \backslash \{0, \infty \}$ (resp. $B_n$), which is periodic under the action of the map $\theta$, is the $x$-coordinate of a rational point of $\Kob (\F_{2^n})$ (resp. $\Kob (\F_{2^{2n}})_{B_n})$, which corresponds to a point of the form $P = (0, P_1, \dots, P_w, P_f) \in R / (\pi_2^n-1) R$ (resp. $R/ (\pi_2^n+1) R$). Each $P_i$, for $1 \leq i \leq w$, has order $p_i^{h_i}$, for some integer $0 \leq h_i \leq e_i$. Moreover, $P_f$ has order $h_f$, for some integer such that $0 \leq h_f \leq f/2$ if $f$ is even or $0 \leq h_f \leq (f+1)/2$ if $f$ is odd.
For any $P_i$ let $l_i$ be the smallest among the  positive integers $k$ such that $[\overline{\pi}_2]^k P_i = P_i$ or $-P_i$.  In a similar way, we define $l_f$ to be the smallest among the positive integers $k$ such that $[\overline{\pi}_2]^k P_f = P_f$ or $-P_f$.
\begin{itemize}
\item If $1 \leq i \leq v$, then $l_i$ is the smallest among the positive integers $k$ such that $p_i^{h_i}$ divides  $\overline{\pi}_2^{k} + 1$ or $\overline{\pi}_2^{k} - 1$ in $R$.
\item If $v+1 \leq i \leq w$, then $l_i$ is the smallest among the positive integers $k$ such that $r_i^{h_i}$ divides  $\overline{\pi}_2^{k} + 1$ or $\overline{\pi}_2^{k} - 1$ in $R$.
\item The integer $l_f$ is the smallest among the positive integers $k$ such that $\sqrt{-7}^{h_f}$ divides  $\overline{\pi}_2^k + 1$ or $\overline{\pi}_2^k - 1$ in $R$.
\end{itemize}

Let
\begin{displaymath}
l = \lcm (l_1, \dots, l_w, l_f).
\end{displaymath}
We introduce parameters $\epsilon_i$, for $1 \leq i \leq w$, and $\epsilon_f$ such defined: 
\begin{displaymath}
\epsilon_i = \left\{
\begin{array}{ll}
1 & \textrm{if $[\overline{\pi}_2]^{l_i} P_i = P_i$}\\
0 & \textrm{if $[\overline{\pi}_2]^{l_i} P_i = - P_i$}.
\end{array}
\right.
\quad
\epsilon_f = \left\{
\begin{array}{ll}
1 & \textrm{if $[\overline{\pi}_2]^{l_f} P_f = P_f$}\\
0 & \textrm{if $[\overline{\pi}_2]^{l_f} P_f = - P_f$}.
\end{array}
\right.
\end{displaymath}
Let 
\begin{displaymath}
\epsilon = \left\{
\begin{array}{ll}
0 & \textrm{if any $\epsilon_i =1$ and $\epsilon_f = 1$ or any $\epsilon_i = 0$ and $\epsilon_f = 0$}\\
1 & \textrm{otherwise}.
\end{array}
\right.
\end{displaymath}
Then, the period of $x$ with respect to $\theta$ is $2^{\epsilon} \cdot l.$

We note that the number of points $P = (0, P_1, \dots, P_w, P_f)$ in $R/(\pi_2^n -1) R$ (resp. $R/(\pi_2^n+1) R$), where each $P_i$ has order $p_i^{h_i}$ and $P_f$ has order $7^{h_f}$, is 
\begin{displaymath}
m = \left( \prod_{i=1}^{v} N_{h_i} \right)  \cdot \left( \prod_{i=v+1}^{w} \phi(p_i^{h_i})  \right) \cdot N_{h_f}
\end{displaymath} 
Let $P = \psi(x,y)$ (resp. $\widetilde{\psi} (x,y)$) be one of such points. The period $l$ of $x$ can be calculated as above. In particular, we note that also $-P = \psi(x,x+y)$ (resp. $\widetilde{\psi} (x,x+y)$) has the same additive order in $R / (\pi_2^n-1)R$ (resp. $R/(\pi_2^n+1) R$).  This amounts to saying that the $m$ points give rise to 
$\left\lceil \dfrac{m}{2l} \right\rceil$    
cycles of length $l$.

Now we define the sets $Z_{e_i} = \{0, 1, \dots, e_i \}$, for any $1 \leq i \leq w$, and $Z_f = \left\{0, 1, \dots, f/2 \right\}$ if $f$ is even or $Z_f = \left\{0, 1, \dots, (f+1)/2 \right\}$ if $f$ is odd. Let
\begin{equation*}
H = \prod_{i=1}^{w} Z_{e_i} \times Z_f.
\end{equation*}

For any $h \in H$ denote by $C_h$ the set of all cycles formed by the elements $x \in A_n$ (resp. $B_n$) such that $(x,y) \in \Kob (\F_{2^n})$ (resp. $\Kob (\F_{2^{2n}})_{B_n}$) for some $y \in \F_{2^n}$ (resp. $\F_{2^{2n}}$) and $\psi(x,y) = P = (0, P_1, \dots, P_w, P_f) \in R / (\pi_2^n-1) R$ (resp. $\widetilde{\psi} (x,y) = P \in R / (\pi_2^n+1) R$), where
\begin{itemize}
\item each $P_i$, for $1 \leq i \leq v$, has additive order $p_i^{h_i}$ in $R / p_i^{e_i} R$;
\item each $P_i$, for $v+1 \leq i \leq w$, has additive order $p_i^{h_i}$ in $R / r_i^{e_i} R$;
\item $P_f$ has additive order $7^{h_f}$ in $R/(\sqrt{-7})^f R$.
\end{itemize}  
Let $l_h$ be the length of the cycles formed by these points. 
Finally, denote by $C_{A_n}$ the set of all cycles in graph $A_n$ and by $C_{B_n}$ the set of all cycles in graph $B_n$.

The following holds.
\begin{lemma}
With the above notation, $C_{A_n}$ (resp. $C_{B_n}$) is equal to $\displaystyle\bigcup_{h \in H} C_h$, being 
\begin{displaymath}
\lvert C_{h} \rvert = \frac{1}{2 l_h} \left( \prod_{i=1}^{v} N_{h_i} \right)  \cdot \left( \prod_{i=v+1}^{w} \phi(p_i^{h_i}) \right) \cdot N_{h_f}
\end{displaymath}
for any non-zero $h \in H$.
\end{lemma} 

In the following we will denote by $V_{A_n}$ (respectively $V_{B_n}$) the set of the elements of $\F_{2^n}$ belonging to some cycle of $C_{A_n}$ (respectively $C_{B_n}$).
Before characterizing the trees rooted in vertices of $V_{A_n}$ (respectively $V_{B_n}$), we observe that
$R / \overline{\pi}_2^{e_0} R $ consists of the  elements
$\displaystyle\sum_{i=0}^{e_0-1} j_i \cdot [\overline{\pi}_2]^{i}$,
where each $j_i$ is $0$ or $1$ (see \cite{gil} for more details).

\subsection{Trees rooted in vertices  of $V_{A_n}$}
The following Lemma characterizes the reversed trees having root in $V_{A_n}$.
\begin{lemma}\label{orb_2}
Any element $x \in V_{A_n}$ is the root of a reversed binary tree of depth $e_0$ with the following properties.
\begin{itemize}
\item If $x \not = \infty$, then there are $\lceil 2^{k-1} \rceil$ vertices at the level $k$ of the tree.  Moreover, the root has one child, while all other vertices have two children.
\item If $x = \infty$, then there are $\lceil 2^{k-2} \rceil$ vertices at the level $k$ of the tree. Moreover, the root and the vertex at the level 1 have one child, while all other vertices have two children. 
\item If $l$ is the greatest power of 2 which divides $n$, then $e_0=l+2$. 
\end{itemize}
\end{lemma}
\begin{proof}
For a fixed element $x \in V_{A_n}$, let $(0, P_1, \dots, P_w, P_f) \in R / (\pi_2^{n}-1) R$ be one of the (at most two) points with such an $x$-coordinate. An element $\tilde{x} \in \F_{2^n}$ belongs to the non-zero level $k$ of the reversed binary tree rooted in $x$ if and only if $\theta^k(\tilde{x}) = x$, $\theta^i (\tilde{x}) \not = x$ and none of the $\theta^i (\tilde{x})$ is periodic for any $i < k$. Since $\tilde{x} \in A_n$ (see Lemma \ref{tr_inverse} and the subsequent Remark), there exists $\tilde{y} \in \F_{2^n}$ such that $(\tilde{x}, \tilde{y}) \in \Kob (\F_{2^n})$ and $\psi(\tilde{x}, \tilde{y}) = Q = (Q_0, Q_1, \dots, Q_w, Q_f)$, where $Q_0 \not =0$. Moreover, since $[\overline{\pi}_2]^{e_0} Q_0 = 0$ in $R / \overline{\pi}_2^{e_0} R$, we have that $k \leq e_0$.

For a fixed positive integer $k \leq e_0$ we aim to find all  points $(Q_0, Q_1,  \dots, Q_w, Q_f)$ in $R /(\pi_2^{n}-1) R $ such that 
\begin{enumerate}
\item $[\overline{\pi}_2]^k Q_0  =  0$ and $[\overline{\pi}_2]^{k-1} Q_0 \not = 0$ ;
\item $[\overline{\pi}_2]^k Q_i  =   P_i$ for any $1 \leq i \leq w$ and $[\overline{\pi}_2]^k Q_f = P_f$, or $[\overline{\pi}_2]^k Q_i  =   -P_i $ for any $1 \leq i \leq w$ and $[\overline{\pi}_2]^k Q_f = -P_f$. 
\end{enumerate}
The first condition is satisfied if and only if
\begin{equation}\label{q0}
Q_0  = [\overline{\pi}_2]^{e_0-k} + \sum_{i=e_0-k+1}^{e_0-1} j_i [\overline{\pi}_2]^i, 
\end{equation}
where each $j_i \in \left\{ 0, 1 \right\}$.
The second condition is satisfied if and only if 
\begin{equation*}
Q_i  =  [\overline{\pi}_2]^{-k} P_i \text{, for any $i$, and } Q_f = [\overline{\pi}_2]^{-k} P_f  
\end{equation*}
or
\begin{equation*}
Q_i  =  -[\overline{\pi}_2]^{-k} P_i \text{, for any $i$, and } Q_f = - [\overline{\pi}_2]^{-k} P_f . 
\end{equation*}

Hence, fixed the values of $j_i$ for $e_0-k+1 \leq i \leq e_0-1$, there are at most two possibilities for $Q$, namely
\begin{eqnarray*}
Q^{(1)} & = & (Q_0, [\overline{\pi}_2]^{-k} P_1, \dots, [\overline{\pi}_2]^{-k} P_w, [\overline{\pi}_2]^{-k} P_f) \quad \textrm{or}  \\
Q^{(2)} & = & (Q_0, -[\overline{\pi}_2]^{-k} P_1, \dots, -[\overline{\pi}_2]^{-k} P_w,-[\overline{\pi}_2]^{-k} P_f) .
\end{eqnarray*}
Therefore, for a fixed positive integer $k$ there are $2^k$ points $Q$, whose $x$-coordinate belongs to the level $k$ of the tree, provided that not all $P_i$ are zero (in which case $ x = \infty$). If $Q = \psi(\tilde{x}, \tilde{y})$ is one of such points, then $-Q=\psi(\tilde{x}, \tilde{x}+\tilde{y})$ has the same $x$-coordinate. Hence, for any $x \in C_{A_n}$ different from $\infty$ there are $2^{k-1}$ vertices at the level $k$ of the reversed binary tree rooted in $x$.

If all $P_i$ are zero (and $x = \infty$), then the points $Q^{(1)}$ and $Q^{(2)}$ coincide. Therefore, for any $k>0$ there are $2^{k-1}$ points, whose $x$-coordinate belongs to the level $k$ of the tree. Moreover, if $Q$ is one of such points, also $-Q$ has the same $x$-coordinate and is different from $Q$, unless $Q$ is the only point belonging to the first level of the tree. This amounts to say that there are $\lceil 2^{k-2} \rceil$ vertices at the level $k$ of the tree.

Consider now an element $\tilde{x}$ belonging to the  level $k < e_0$ of the tree rooted in some $x \in V_{A_n}$. Such an $\tilde{x}$ is the $x$-coordinate of a point $Q=(Q_0, Q_1, \dots, Q_w, Q_f)$ in $R/(\pi_2^n-1) R$, for some $Q_0$ as in (\ref{q0}) or $Q_0 = 0$. The equation $z+z^{-1} = \tilde{x}$ is satisfied for at most two $z$ in $\F_{2^n}$, which are the $x$-coordinate of two points in $R/(\pi_2^n-1)R$,
\begin{equation*}
\widetilde{Q}^{(1)} = (\widetilde{Q}_0^{(1)}, \widetilde{Q}_1, \dots, \widetilde{Q}_w, \widetilde{Q}_f) \textrm{ and } \widetilde{Q}^{(2)} = (\widetilde{Q}_0^{(2)}, \widetilde{Q}_1, \dots, \widetilde{Q}_w, \widetilde{Q}_f),
\end{equation*}
where 
\begin{displaymath}
\begin{array}{lll}
\widetilde{Q}_0^{(1)} & = & [\overline{\pi}_2]^{e_0-k-1} + \displaystyle\sum_{i=e_0-k}^{e_0-2} j_{i+1} [\overline{\pi}_2]^i \\
\widetilde{Q}_0^{(2)} & = & [\overline{\pi}_2]^{e_0-k-1} + \displaystyle\sum_{i=e_0-k}^{e_0-2} j_{i+1} [\overline{\pi}_2]^i + [\overline{\pi}_2]^{e_0-1}\\
\widetilde{Q}_i  & = &  [\overline{\pi}_2]^{-1} Q_i, \quad \text {if $1 \leq i \leq w$}\\
\widetilde{Q}_f  & = &  [\overline{\pi}_2]^{-1} Q_f
\end{array}
\end{displaymath}
and $[\overline{\pi}_2] \widetilde{Q}^{(1)} = [\overline{\pi}_2] \widetilde{Q}^{(2)} = Q$.

If $k=0$, then just the point $\widetilde{Q}^{(1)}$ belongs to the tree, proving that the root of the tree has one and only one child.

If $k \geq 1$ and at least one of the $P_i$ is non-zero, then $\widetilde{Q}^{(1)} \not = - \widetilde{Q}^{(2)}$, hence $\widetilde{Q}^{(1)}$ and $\widetilde{Q}^{(2)}$ have different $x$-coordinates. This implies that each vertex at non-zero level $k$ of a tree rooted in $x \in V_{A_n} \backslash \{ \infty \}$ has two children.

Suppose now that all the $P_i$ are zero. In this case $x= \infty$.
If $k = 1$, then $\widetilde{Q}_0^{(1)} = - \widetilde{Q}_0^{(2)}$ and also $\widetilde{Q}^{(1)} = - \widetilde{Q}^{(2)}$. Hence the only vertex at the level 1 has exactly one child.

Finally, if $k > 1$ and  all the $P_i$ are zero, then $\widetilde{Q}^{(1)} \not = - \widetilde{Q}^{(2)}$. Hence, each of the vertices at the levels $k>1$ of the tree rooted in $\infty$ has two children.

As regards the number $e_0$, we note that $N(\pi_2^n-1) = \lvert \Kob (\F_{2^n}) \rvert$. We remind that 
\begin{equation*}
\lvert \Kob(\F_{2^n}) \rvert = 2^n+1+S^{(n)} (1)
\end{equation*}
and that $R/(\pi_2^n-1) R$ is isomorphic to a product of rings as in (\ref{orb4}). 
Since $\pi_2 \overline{\pi}_2 = 2$ in $R$, then ${e_0}$ is the greatest power of 2 which divides $\lvert \Kob (\F_{2^n}) \rvert$.
We want to prove that, if $n=2^l m$, for some odd $m$, then $e_0 = l+2$. Firstly we consider the cases $n=2$ and $n=4$. 

If $n=2$, and $l=1$, then $S^{(2)} (1) = 3$ and $\lvert \Kob (\F_{2^2}) \rvert = 8 = 2^3$. Hence $e_0 = 3 = l+2$.

If $n=4$, and $l=2$,  then $S^{(4)} (1) = -1$ and $\lvert \Kob (\F_{2^4}) \rvert = 16 = 2^4$. Hence $e_0 = 4 = l+2$.

Now consider a positive integer $n \not \in \{ 2, 4 \}$ greater than 1. Then, as a consequence of Lemma \ref{Sncon}
\begin{displaymath}
\lvert \Kob (\F_{2^n}) \rvert \equiv
\left\{
\begin{array}{ll}
0 & \pmod{2^{l+2}}\\
2^{l+2} & \pmod{2^{l+3}}.
\end{array}
\right.
\end{displaymath}
\end{proof}

\subsection{Trees rooted in elements of $V_{B_n}$}
The following Lemma characterizes the reversed trees having root in $V_{B_n}$.
\begin{lemma}\label{orb_3}
Any element $x \in V_{B_n}$ is the root of a reversed binary tree of depth 1 and has one child.
\end{lemma}
\begin{proof}
We remind that the elements of $B_n$ are the $x$-coordinates of the points in $\Kob (\F_{2^{2n}})_{B_n} \backslash \{(0,1), O \}$. Moreover there is an isomorphism
\begin{eqnarray*}
\widetilde{\psi} : \Kp (\F_{2^{2n}})_{B_n} & \to & R / (\pi_2^{n} + 1) R.
\end{eqnarray*}
We want to prove that the greatest power of $\overline{\pi}_2$ which divides $\pi_2^n+1$ is $1$. Since $\pi_2 \overline{\pi}_2 = 2$ in $R$, then the greatest power of $\overline{\pi}_2$ which divides $\pi_2^n+1$ is the greatest power of $2$ which divides $N(\pi_2^n+1)$. We consider the fact that
\begin{equation*}
N(\pi_2^{2n}-1) = N(\pi_2^n-1) N(\pi_2^n+1).
\end{equation*}
We know that
\begin{eqnarray*}
N(\pi_2^n-1) & = & 2^n + 1 +  S^{(n)} (1)\\
N(\pi_2^{2n}-1) & = & 2^{2n} + 1 + S^{(2n)} (1). 
\end{eqnarray*}
Now it is an easy matter to check that
\begin{equation*}
N(\pi_2^n+1) = 2^n+1-S^{(n)} (1). 
\end{equation*}
Since $S^{(n)} (1) \equiv -1 \pmod{4}$, for $n \geq 2$,  then
\begin{equation*}
N(\pi_2^n+1) \equiv 2 \pmod{4}
\end{equation*}
and we are done. Note in passing that for $n=1$ the field $\F_{2^n} = \F_{2}$ and there are no elements $x \in \F_{2}^*$ such that $\Tr_1 (x) \not = \Tr_1 (x^{-1})$.

Hence, $R/(\pi_2^n+1) R$ is isomorphic to a product of ring as in (\ref{orb4}), where $e_0 = 1$. This implies that any $x \in V_{B_n}$ is the root of a tree of depth 1. Consider now a point $P=(0,P_1, \dots, P_w) \in R/(\pi_2^n+1) R$. The only point $Q$ such that $[\overline{\pi}_2] Q = P$ is $Q=(1, [\overline{\pi}_2]^{-1} P_1, \dots, [\overline{\pi}_2]^{-1} P_w)$. Hence,  any $x \in V_{B_n}$ is the root of a tree of depth 1 and has one child.
\end{proof}

\section{Examples}
\subsection{Graph associated with $\theta$ in $\mathbf{\F_{2^{5}}}$.}
At the beginning of this paper we constructed explicitly the graph over the field $\F_{2^5}$.

The structure of the graph is as follows:
\begin{center}
\begin{tabular}{|c|c|c|}
\hline
\multicolumn{3}{|c|}{Graph $A_{5}$}\\
\hline
Length of the cycles & Number of cycles & Depth of the trees \\
\hline
1 & 1 & 2\\
5 & 1 & 2\\
\hline
\end{tabular}
\end{center}

\begin{center}
\begin{tabular}{|c|c|c|}
\hline
\multicolumn{3}{|c|}{Graph $B_{5}$}\\
\hline
Length of the cycles & Number of cycles & Depth of the trees \\
\hline
5 & 1 & 1\\
\hline
\end{tabular}
\end{center}

Let us analyse separately graphs $A_5$ and $B_5$.

\textbf{Graph $A_{5}$.}
The rational points of $\Kob (\F_{2^5})$ different from the point at infinity are all of the form $(x,y)$, for some $x \in A_{5}$, and, for any non-zero $x \in A_{5}$, there are exactly two elements $y_1, y_2 \in \F_{2^5}$ such that $(x, y_1), (x, y_2)$ are points of $\Kob (\F_{2^5})$. Moreover, there is a $1-1$ correspondence between $\Kob (\F_{2^5})$ and $R / (\pi_2^{5}-1) R$, where $R = \Z[\omega]$. 

The factorization of $\pi_2^5-1$ in primes of $R$ is
\begin{equation*}
\pi_2^{5}-1 = \overline{\pi}_2^2 \cdot (1+2 \omega).
\end{equation*}
If we denote $p_1 = 1+2 \omega$, then 
\begin{equation*}
R / (\pi_2^{5}-1) R \cong R / \overline{\pi}_2^2 R \times R /  p_1 R.
\end{equation*}

Consider the points $P = (0, r) \in R / \overline{\pi}_2^2 R \times R /  p_1 R$, where $r \in R$ is not divisible by $1+2 \omega$. Such points have additive order $11$ in $R/p_1 R$. The integer $l_r = 5$ is the smallest among the  positive integers $k$ such that $[\overline{\pi}_2]^{k} P =  P$ or $-P$, namely the smallest among the positive integers $k$ such that $p_1 \mid (\overline{\pi}_2^{k} + 1)$ or $(\overline{\pi}_2^{k}-1)$. In particular, $p_1 \mid (\overline{\pi}_2^{5} + 1)$. There are 10 points of this form, corresponding to 5 different values of $x \in \F_{2^5}$. Hence such points give rise to just one cycle of length 5. Moreover, each of these nodes is the root of a reversed binary tree of depth 2, since 2 is the greatest power of $\overline{\pi}_2$ which divides $\pi_2^5-1$. Finally, consider the point $O=(0,0)$. The corresponding node is labelled by $\infty$, belongs to a cycle of length one, namely a loop, and is the root of a binary tree of depth 2.

\textbf{Graph $B_{5}$.}
The nodes belonging to the graph $B_5$ are the $x$-coordinate of points of $\Kob (\F_{2^{10}})$ such that $x \in \F_{2^5}$ and $\Tr_5 (x) \not = \Tr_5 (x^{-1})$. We remind the isomorphism
\begin{equation*}
\Kob (\F_{2^{10}})_{B_{5}} \to R / (\pi_2^5+1) R.
\end{equation*}
We have that
\begin{equation*}
 R / (\pi_2^5+1) R \cong R/\overline{\pi}_2 R \times R/(3-2 \omega) R.
\end{equation*}

Consider the points $P=(0,r)$ where $r \not = 0$ in $R/(3-2\omega)R$. Hence, $3-2 \omega$ does not divide $r$. The integer $l_r = 5$ is the smallest among the  positive integers $k$ such that $[\overline{\pi}_2]^{k} P =  P$ or $-P$. Since there are 10 points $P$ of this form, corresponding to 5 different $x$ in $\F_{2^5}$, then there is just a cycle of length 5 in graph $B_5$. Moreover each of the nodes of the cycle is the root of a reversed binary tree of depth 1.

\subsection{Graph associated with $\theta$ over $\mathbf{\F_{2^{8}}}$.}
We can construct the field with $2^8$ elements as the splitting field of the primitive polynomial $x^8+x^4+x^3+x^2+1$ over $\F_2$. If $\alpha$ is a root of this polynomial in $\F_{2^8}$, then $\Pro (\F_{2^8}) =\left\{0 \right\} \cup \left\{\alpha^i: 1 \leq i \leq 255 \right\} \cup \left\{ \infty \right\}$.The graph is formed by the following three connected components.

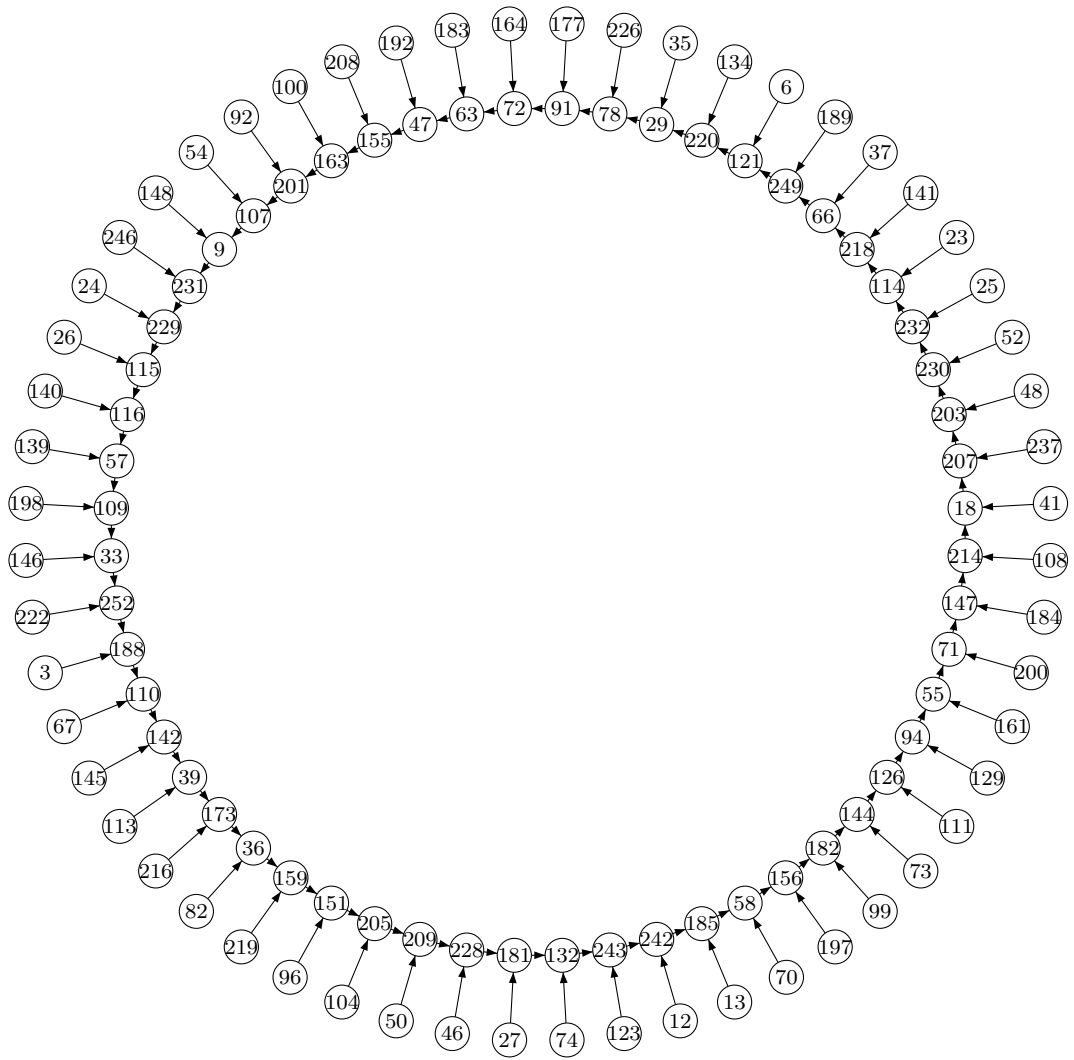
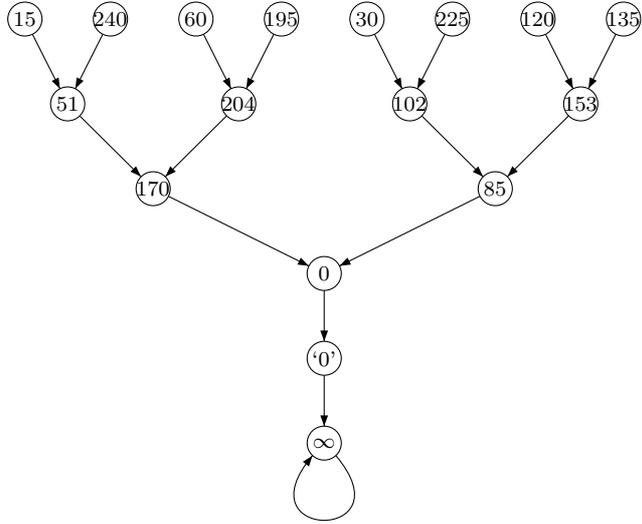
\begin{center}
    \unitlength=3.2pt
    \begin{picture}(40, 120)(-20,-20)
    \gasset{Nw=4,Nh=4,Nmr=3,curvedepth=-2.5}
    \thinlines
   \tiny
    \node(A1)(10,0){$119$}
    \node(A2)(0,10){187}
    \node(A3)(-10,0){221}          
    \node(A4)(0,-10){238}
     
    \drawedge(A1,A2){}
    \drawedge(A2,A3){}
    \drawedge(A3,A4){}
    \drawedge(A4,A1){}
     
    \node(B1)(20,0){17}
    \node(B2)(0,20){136}
    \node(B3)(-20,0){68}          
    \node(B4)(0,-20){34}
    \gasset{curvedepth=0}
    \drawedge(B1,A1){}
    \drawedge(B2,A2){}
    \drawedge(B3,A3){}
    \drawedge(B4,A4){}
    
    \node(C1)(27.7,11.48){210}
    \node(C2)(11.48,27.7){105}
    \node(C3)(-11.48,27.7){150}          
    \node(C4)(-27.7,11.48){75}
    \node(C5)(-27.7,-11.48){180}
    \node(C6)(-11.48,-27.7){90}
    \node(C7)(11.48,-27.7){165}          
    \node(C8)(27.7,-11.48){45}
    
    \drawedge(C1,B1){}
    \drawedge(C8,B1){}
    \drawedge(C2,B2){}
    \drawedge(C3,B2){}
    \drawedge(C4,B3){}
    \drawedge(C5,B3){}
    \drawedge(C6,B4){}
    \drawedge(C7,B4){}
    
    \node(D1)(39.23,7.8){118}
    \node(D2)(33.25,22.2){137}
    \node(D3)(22.22,33.25){59}
    \node(D4)(7.8,39.23){196}          
    \node(D5)(-7.8,39.23){76}
    \node(D6)(-22.22,33.25){179}
    \node(D7)(-33.25,22.2){38}          
    \node(D8)(-39.23,7.8){217}

    \drawedge(D1,C1){}
    \drawedge(D2,C1){}
    \drawedge(D3,C2){}
    \drawedge(D4,C2){}
    \drawedge(D5,C3){}
    \drawedge(D6,C3){}
    \drawedge(D7,C4){}
    \drawedge(D8,C4){}    
    
    \node(D11)(39.23,-7.8){103}
    \node(D12)(33.25,-22.2){152}
    \node(D13)(22.22,-33.25){19}
    \node(D14)(7.8,-39.23){236}          
    \node(D15)(-7.8,-39.23){206}
    \node(D16)(-22.22,-33.25){49}
    \node(D17)(-33.25,-22.2){98}          
    \node(D18)(-39.23,-7.8){157}

  	\drawedge(D11,C8){}
    \drawedge(D12,C8){}
    \drawedge(D13,C7){}
    \drawedge(D14,C7){}
    \drawedge(D15,C6){}
    \drawedge(D16,C6){}
    \drawedge(D17,C5){}
    \drawedge(D18,C5){}    
    
    \node(E1)(49.75,4.9){62}
    \node(E2)(47.85,14.51){193}
    \node(E3)(44.1,23.57){247}
    \node(E4)(38.65,31.71){8}          
    \node(E5)(31.71,38.65){31}
    \node(E6)(23.57,44.1){224}
    \node(E7)(14.51,47.85){251}          
    \node(E8)(4.9,49.75){4}
    
    \drawedge(E1,D1){}
    \drawedge(E2,D1){}
    \drawedge(E3,D2){}
    \drawedge(E4,D2){}
    \drawedge(E5,D3){}
    \drawedge(E6,D3){}
    \drawedge(E7,D4){}
    \drawedge(E8,D4){} 
    
    \node(E11)(-49.75,4.9){7}
    \node(E12)(-47.85,14.51){248}
    \node(E13)(-44.1,23.57){32}
    \node(E14)(-38.65,31.71){223}          
    \node(E15)(-31.71,38.65){14}
    \node(E16)(-23.57,44.1){241}
    \node(E17)(-14.51,47.85){64}          
    \node(E18)(-4.9,49.75){191}
    
    \drawedge(E18,D5){}
    \drawedge(E17,D5){}
    \drawedge(E16,D6){}
    \drawedge(E15,D6){}
    \drawedge(E14,D7){}
    \drawedge(E13,D7){}
    \drawedge(E12,D8){}
    \drawedge(E11,D8){} 
    
    \node(E21)(49.75,-4.9){28}
    \node(E22)(47.85,-14.51){227}
    \node(E23)(44.1,-23.57){128}
    \node(E24)(38.65,-31.71){127}          
    \node(E25)(31.71,-38.65){16}
    \node(E26)(23.57,-44.1){239}
    \node(E27)(14.51,-47.85){131}          
    \node(E28)(4.9,-49.75){124}
    
     \drawedge(E21,D11){}
     \drawedge(E22,D11){}
    \drawedge(E23,D12){}
    \drawedge(E24,D12){}
    \drawedge(E25,D13){}
    \drawedge(E26,D13){}
    \drawedge(E27,D14){}
    \drawedge(E28,D14){} 
    
    \node(E31)(-49.75,-4.9){112}
    \node(E32)(-47.85,-14.51){143}
    \node(E33)(-44.1,-23.57){253}
    \node(E34)(-38.65,-31.71){2}          
    \node(E35)(-31.71,-38.65){1}
    \node(E36)(-23.57,-44.1){254}
    \node(E37)(-14.51,-47.85){56}          
    \node(E38)(-4.9,-49.75){199}
    
     \drawedge(E31,D18){}
     \drawedge(E32,D18){}
    \drawedge(E33,D17){}
    \drawedge(E34,D17){}
    \drawedge(E35,D16){}
    \drawedge(E36,D16){}
    \drawedge(E37,D15){}
    \drawedge(E38,D15){} 
    
    \node(F1)(59.93,2.95){89}
    \node(F2)(59.35,8.8){166}
    \node(F3)(58.2,14.58){22}
    \node(F4)(56.49,20.21){233}          
    \node(F5)(54.24,25.65){168}
    \node(F6)(51.46,30.84){87}
    \node(F7)(48.19,35.74){130}          
    \node(F8)(44.46,40.29){125}

     \drawedge(F1,E1){}
     \drawedge(F2,E1){}
    \drawedge(F3,E2){}
    \drawedge(F4,E2){}
    \drawedge(F5,E3){}
    \drawedge(F6,E3){}
    \drawedge(F7,E4){}
    \drawedge(F8,E4){} 
    
    \node(F11)(40.29,44.46){83}
    \node(F12)(35.74,48.19){172}
    \node(F13)(30.84,51.46){11}
    \node(F14)(25.65,54.24){244}          
    \node(F15)(20.21,56.49){84}
    \node(F16)(14.58,58.2){171}
    \node(F17)(8.8,59.35){190}          
    \node(F18)(2.95,59.93){65}
    
     \drawedge(F11,E5){}
     \drawedge(F12,E5){}
    \drawedge(F13,E6){}
    \drawedge(F14,E6){}
    \drawedge(F15,E7){}
    \drawedge(F16,E7){}
    \drawedge(F17,E8){}
    \drawedge(F18,E8){} 
    
 	\node(F21)(-59.93,2.95){167}
    \node(F22)(-59.35,8.8){88}
    \node(F23)(-58.2,14.58){154}
    \node(F24)(-56.49,20.21){101}          
    \node(F25)(-54.24,25.65){245}
    \node(F26)(-51.46,30.84){10}
    \node(F27)(-48.19,35.74){162}          
    \node(F28)(-44.46,40.29){93}
    
    \drawedge(F21,E11){}
    \drawedge(F22,E11){}
    \drawedge(F23,E12){}
    \drawedge(F24,E12){}
    \drawedge(F25,E13){}
    \drawedge(F26,E13){}
    \drawedge(F27,E14){}
    \drawedge(F28,E14){} 

    \node(F31)(-40.29,44.46){79}
    \node(F32)(-35.74,48.19){176}
    \node(F33)(-30.84,51.46){53}
    \node(F34)(-25.65,54.24){202}          
    \node(F35)(-20.21,56.49){20}
    \node(F36)(-14.58,58.2){235}
    \node(F37)(-8.8,59.35){69}          
    \node(F38)(-2.95,59.93){186}   
    
    \drawedge(F31,E15){}
    \drawedge(F32,E15){}
    \drawedge(F33,E16){}
    \drawedge(F34,E16){}
    \drawedge(F35,E17){}
    \drawedge(F36,E17){}
    \drawedge(F37,E18){}
    \drawedge(F38,E18){}

    \node(F41)(59.93,-2.95){158}
    \node(F42)(59.35,-8.8){97}
    \node(F43)(58.2,-14.58){149}
    \node(F44)(56.49,-20.21){106}          
    \node(F45)(54.24,-25.65){215}
    \node(F46)(51.46,-30.84){40}
    \node(F47)(48.19,-35.74){138}          
    \node(F48)(44.46,-40.29){117}
    
     \drawedge(F41,E21){}
    \drawedge(F42,E21){}
    \drawedge(F43,E22){}
    \drawedge(F44,E22){}
    \drawedge(F45,E23){}
    \drawedge(F46,E23){}
    \drawedge(F47,E24){}
    \drawedge(F48,E24){}  

    \node(F51)(40.29,-44.46){250}
    \node(F52)(35.74,-48.19){5}
    \node(F53)(30.84,-51.46){174}
    \node(F54)(25.65,-54.24){81}          
    \node(F55)(20.21,-56.49){211}
    \node(F56)(14.58,-58.2){44}
    \node(F57)(8.8,-59.35){178}          
    \node(F58)(2.95,-59.93){77}
    
    \drawedge(F51,E25){}
    \drawedge(F52,E25){}
    \drawedge(F53,E26){}
    \drawedge(F54,E26){}
    \drawedge(F55,E27){}
    \drawedge(F56,E27){}
    \drawedge(F57,E28){}
    \drawedge(F58,E28){}  
    
 	\node(F61)(-59.93,-2.95){122}
    \node(F62)(-59.35,-8.8){133}
    \node(F63)(-58.2,-14.58){86}
    \node(F64)(-56.49,-20.21){169}          
    \node(F65)(-54.24,-25.65){42}
    \node(F66)(-51.46,-30.84){213}
    \node(F67)(-48.19,-35.74){95}          
    \node(F68)(-44.46,-40.29){160}
    
    \drawedge(F61,E31){}
    \drawedge(F62,E31){}
    \drawedge(F63,E32){}
    \drawedge(F64,E32){}
    \drawedge(F65,E33){}
    \drawedge(F66,E33){}
    \drawedge(F67,E34){}
    \drawedge(F68,E34){} 

    \node(F71)(-40.29,-44.46){80}
    \node(F72)(-35.74,-48.19){175}
    \node(F73)(-30.84,-51.46){21}
    \node(F74)(-25.65,-54.24){234}          
    \node(F75)(-20.21,-56.49){61}
    \node(F76)(-14.58,-58.2){194}
    \node(F77)(-8.8,-59.35){43}          
    \node(F78)(-2.95,-59.93){212}
    
     \drawedge(F71,E35){}
    \drawedge(F72,E35){}
    \drawedge(F73,E36){}
    \drawedge(F74,E36){}
    \drawedge(F75,E37){}
    \drawedge(F76,E37){}
    \drawedge(F77,E38){}
    \drawedge(F78,E38){} 
 
    \end{picture}
  \end{center}
  
  \begin{center}
    \unitlength=3.2pt
    \begin{picture}(40, 140)(-20,-40)
    \gasset{Nw=4,Nh=4,Nmr=3,curvedepth=0}
    \tiny
    \node(A1)(59.90,3.36){41}
    \node(B1)(49.92,2.8){18}
    
    \node(A2)(59.15,10.05){237}
    \node(B2)(49.29,8.37){207}
    
    \node(A3)(57.65,16.61){48}          
    \node(B3)(48.04,13.84){203}

   	\node(A4)(55.43,22.96){52}
    \node(B4)(46.19,19.13){230}

    \node(A5)(52.51,29.02){25}          
    \node(B5)(43.76,24.18){232}

	\node(A6)(48.93,34.72){23}
    \node(B6)(40.78,28.93){114}
    
    \node(A7)(44.73,39.98){141}          
    \node(B7)(37.28,33.31){218}

    \node(A8)(39.98,44.74){37}
    \node(B8)(33.31,37.28){66}
    
    \node(A9)(34.72,48.93){189}
    \node(B9)(28.93,40.78){249}
    
	\node(A10)(29.02, 52.51){6}          
    \node(B10)(24.18, 43.76){121}  
    
    \node(A11)(22.96,55.43){134}
    \node(B11)(19.13,46.19){220} 
    
    \node(A12)(16.61,57,65){35}          
    \node(B12)(13.84,48.04){29}
    
    \node(A13)(10.05,59.15){226}
    \node(B13)(8.37,49.29){78}
    
    \node(A14)(3.36,59.90){177}
    \node(B14)(2.8,49.92){91}
    
    \node(A15)(-59.90,3.36){198}
    \node(B15)(-49.92,2.8){109}
     
    \node(A16)(-59.15,10.05){139}
    \node(B16)(-49.29,8.37){57}
    
    \node(A17)(-57.65,16.61){140}          
    \node(B17)(-48.04,13.84){116}

   	\node(A18)(-55.43,22.96){26}
    \node(B18)(-46.19,19.13){115}

    \node(A19)(-52.51,29.02){24}          
    \node(B19)(-43.76,24.18){229}

	\node(A20)(-48.93,34.72){246}
    \node(B20)(-40.78,28.93){231}
    
    \node(A21)(-44.73,39.98){148}          
    \node(B21)(-37.28,33.31){9}

    \node(A22)(-39.98,44.74){54}
    \node(B22)(-33.31,37.28){107}
    
    \node(A23)(-34.72,48.93){92}
    \node(B23)(-28.93,40.78){201}
    
	\node(A24)(-29.02, 52.51){100}          
    \node(B24)(-24.18, 43.76){163}  
    
    \node(A25)(-22.96,55.43){208}
    \node(B25)(-19.13,46.19){155} 
    
    \node(A26)(-16.61,57,65){192}          
    \node(B26)(-13.84,48.04){47}
    
    \node(A27)(-10.05,59.15){183}
    \node(B27)(-8.37,49.29){63}
    
    \node(A28)(-3.36,59.90){164}
    \node(B28)(-2.8,49.92){72}
    
    \node(A101)(59.90,-3.36){108}
    \node(B101)(49.92,-2.8){214}
    
    \node(A102)(59.15,-10.05){184}
    \node(B102)(49.29,-8.37){147}
    
    \node(A103)(57.65,-16.61){200}          
    \node(B103)(48.04,-13.84){71}

   	\node(A104)(55.43,-22.96){161}
    \node(B104)(46.19,-19.13){55}

    \node(A105)(52.51,-29.02){129}          
    \node(B105)(43.76,-24.18){94}

	\node(A106)(48.93,-34.72){111}
    \node(B106)(40.78,-28.93){126}
    
    \node(A107)(44.73,-39.98){73}          
    \node(B107)(37.28,-33.31){144}

    \node(A108)(39.98,-44.74){99}
    \node(B108)(33.31,-37.28){182}
    
    \node(A109)(34.72,-48.93){197}
    \node(B109)(28.93,-40.78){156}
    
	\node(A110)(29.02, -52.51){70}          
    \node(B110)(24.18, -43.76){58}  
    
    \node(A111)(22.96,-55.43){13}
    \node(B111)(19.13,-46.19){185} 
    
    \node(A112)(16.61,-57,65){12}          
    \node(B112)(13.84,-48.04){242}
    
    \node(A113)(10.05,-59.15){123}
    \node(B113)(8.37,-49.29){243}
    
    \node(A114)(3.36,-59.90){74}
    \node(B114)(2.8,-49.92){132}
    
    \node(A115)(-59.90,-3.36){146}
    \node(B115)(-49.92,-2.8){33}
    
    \node(A116)(-59.15,-10.05){222}
    \node(B116)(-49.29,-8.37){252}
    
    \node(A117)(-57.65,-16.61){3}          
    \node(B117)(-48.04,-13.84){188}

   	\node(A118)(-55.43,-22.96){67}
    \node(B118)(-46.19,-19.13){110}

    \node(A119)(-52.51,-29.02){145}          
    \node(B119)(-43.76,-24.18){142}

	\node(A120)(-48.93,-34.72){113}
    \node(B120)(-40.78,-28.93){39}
    
    \node(A121)(-44.73,-39.98){216}          
    \node(B121)(-37.28,-33.31){173}

    \node(A122)(-39.98,-44.74){82}
    \node(B122)(-33.31,-37.28){36}
    
    \node(A123)(-34.72,-48.93){219}
    \node(B123)(-28.93,-40.78){159}
    
	\node(A124)(-29.02, -52.51){96}          
    \node(B124)(-24.18, -43.76){151}  
    
    \node(A125)(-22.96,-55.43){104}
    \node(B125)(-19.13,-46.19){205} 
    
    \node(A126)(-16.61,-57,65){50}          
    \node(B126)(-13.84,-48.04){209}
    
    \node(A127)(-10.05,-59.15){46}
    \node(B127)(-8.37,-49.29){228}
    
    \node(A128)(-3.36,-59.90){27}
    \node(B128)(-2.8,-49.92){181}
	
	\drawedge(B1,B2){}	
    \drawedge(B2,B3){}
    \drawedge(B3,B4){}
    \drawedge(B4,B5){}
    \drawedge(B5,B6){}
    \drawedge(B6,B7){}
    \drawedge(B7,B8){}
    \drawedge(B8,B9){}
    \drawedge(B9,B10){}
    \drawedge(B10,B11){}
    \drawedge(B11,B12){}
    \drawedge(B12,B13){}
    \drawedge(B13,B14){}
    \drawedge(B14,B28){}
    \drawedge(B28,B27){}
    \drawedge(B27,B26){}
    \drawedge(B26,B25){}
    \drawedge(B25,B24){}
    \drawedge(B24,B23){}
    \drawedge(B23,B22){}
    \drawedge(B22,B21){}
    \drawedge(B21,B20){}
    \drawedge(B20,B19){}
    \drawedge(B19,B18){}
    \drawedge(B18,B17){}
    \drawedge(B17,B16){}
    \drawedge(B16,B15){}
    \drawedge(B101,B1){}
    \drawedge(B102,B101){}
    \drawedge(B103,B102){}
    \drawedge(B104,B103){}
    \drawedge(B105,B104){}
    \drawedge(B106,B105){}
    \drawedge(B107,B106){}
    \drawedge(B108,B107){}
    \drawedge(B109,B108){}
    \drawedge(B110,B109){}
    \drawedge(B111,B110){}
    \drawedge(B112,B111){}
    \drawedge(B113,B112){}
    \drawedge(B114,B113){}
    \drawedge(B115,B116){}
    \drawedge(B116,B117){}
    \drawedge(B117,B118){}
    \drawedge(B118,B119){}
    \drawedge(B119,B120){}
    \drawedge(B120,B121){}
    \drawedge(B121,B122){}
    \drawedge(B122,B123){}
    \drawedge(B123,B124){}
    \drawedge(B124,B125){}
    \drawedge(B125,B126){}
    \drawedge(B126,B127){}
    \drawedge(B127,B128){}
    \drawedge(B128,B114){}
    \drawedge(B15,B115){}
    
    \drawedge(A1,B1){}
    \drawedge(A2,B2){}
    \drawedge(A3,B3){}
    \drawedge(A4,B4){}
    \drawedge(A5,B5){}
    \drawedge(A6,B6){}
    \drawedge(A7,B7){}
    \drawedge(A8,B8){}
    \drawedge(A9,B9){}
    \drawedge(A10,B10){}
    \drawedge(A11,B11){}
    \drawedge(A12,B12){}
    \drawedge(A13,B13){}
    \drawedge(A14,B14){}
    \drawedge(A15,B15){}
    \drawedge(A16,B16){}
    \drawedge(A17,B17){}
    \drawedge(A18,B18){}
    \drawedge(A19,B19){}
    \drawedge(A20,B20){}
    \drawedge(A21,B21){}
    \drawedge(A22,B22){}
    \drawedge(A23,B23){}
    \drawedge(A24,B24){}
	\drawedge(A25,B25){}
    \drawedge(A26,B26){}
    \drawedge(A27,B27){}
    \drawedge(A28,B28){}
    \drawedge(A101,B101){}
    \drawedge(A102,B102){}
    \drawedge(A103,B103){}
    \drawedge(A104,B104){}
    \drawedge(A105,B105){}
    \drawedge(A106,B106){}
    \drawedge(A107,B107){}
    \drawedge(A108,B108){}
    \drawedge(A109,B109){}
    \drawedge(A110,B110){}
    \drawedge(A111,B111){}
    \drawedge(A112,B112){}
    \drawedge(A113,B113){}
    \drawedge(A114,B114){}
    \drawedge(A115,B115){}
    \drawedge(A116,B116){}
    \drawedge(A117,B117){}
    \drawedge(A118,B118){}
    \drawedge(A119,B119){}
    \drawedge(A120,B120){}
    \drawedge(A121,B121){}
    \drawedge(A122,B122){}
    \drawedge(A123,B123){}
    \drawedge(A124,B124){}
	\drawedge(A125,B125){}
    \drawedge(A126,B126){}
    \drawedge(A127,B127){}
    \drawedge(A128,B128){}
   
    \end{picture}
  \end{center}
  
   \begin{center}
    \unitlength=3.2pt
    \begin{picture}(80, 65)(0,-10)
    \gasset{Nw=4,Nh=4,Nmr=3,curvedepth=0}
    \thinlines
   \tiny
    \node(A1)(45,0){$\infty$}
    
    \node(B1)(45,10){`0'}
    
    \node(C1)(45,20){0}
    
    \node(D1)(25,30){170}
    \node(D2)(65,30){85}
              
    \node(E1)(15,40){51}
   	\node(E2)(35,40){204}
    \node(E3)(55,40){102}
    \node(E4)(75,40){153}

	\node(F1)(10,50){15}
    \node(F2)(20,50){240}
    \node(F3)(30,50){60}          
    \node(F4)(40,50){195}
    \node(F5)(50,50){30}
    \node(F6)(60,50){225}
    \node(F7)(70,50){120}
    \node(F8)(80,50){135}

	\drawedge(F1,E1){}	
    \drawedge(F2,E1){}
    \drawedge(F3,E2){}
    \drawedge(F4,E2){}
    \drawedge(F5,E3){}
    \drawedge(F6,E3){}
    \drawedge(F7,E4){}
    \drawedge(F8,E4){}
    \drawedge(E1,D1){}
    \drawedge(E2,D1){}
    \drawedge(E3,D2){}
    \drawedge(E4,D2){}
    \drawedge(D1,C1){}
    \drawedge(D2,C1){}
    \drawedge(C1,B1){}
    \drawedge(B1,A1){}
    \drawloop[loopangle=-90](A1){} 
    \end{picture}
  \end{center}

The structure of the graph is summarized in the following tables.
\begin{center}
\begin{tabular}{|c|c|c|}
\hline
\multicolumn{3}{|c|}{Graph $A_{8}$}\\
\hline
Length of the cycles & Number of cycles & Depth of the trees \\
\hline
1 & 1 & 5\\
4 & 1 & 5\\
\hline
\end{tabular}
\end{center}

\begin{center}
\begin{tabular}{|c|c|c|}
\hline 
\multicolumn{3}{|c|}{Graph $B_{8}$}\\
\hline
Length of the cycles & Number of cycles & Depth of the trees \\
\hline
56 & 1 & 1\\
\hline
\end{tabular}
\end{center}
\textbf{Graph $A_{8}$.}
The points belonging to $\Kob (\F_{2^{8}})$ different from the point at infinity are all of the form $(x,y)$, for some $x \in A_{8}$, and, for any non-zero $x \in A_{8}$, there are exactly two elements $y_1, y_2 \in \F_{2^{8}}$ such that $(x, y_1), (x, y_2)$ are points of $\Kob (\F_{2^{8}})$. Moreover, there is a bijection between $\Kob (\F_{2^{8}})$ and $R / (\pi_2^{8}-1) R$, where $R = \Z[\omega]$.

After factorization in primes of $R$ we get that
\begin{equation*}
{\pi}_2^8-1 = \overline{\pi}_2^5 \cdot 3.
\end{equation*}

Hence,
\begin{equation*}
R / (\pi_2^8-1) R \cong R / \overline{\pi}_2^5 R \times R / 3 R.
\end{equation*}

Consider the points $P$ of the form $(0,r)$, being $r$ an element of $R/3R$ different from zero. The additive order of $r$ in $R/3R$ is $3$. The integer $l_r = 4$ is the smallest among the  positive integers $k$ such that $[\overline{\pi}_2]^{k} P =  P$ or $-P$. In particular, $3 \mid (\overline{\pi}_2^{4}+1)$. Hence, the $x$-coordinate of each of these points belongs to a cycle of length $4$. Since there are 8 points of this form, corresponding to $4$ different elements $x \in \F_{2^8}$, then these 8 points give rise to just one cycle of length 4. Moreover, being $5$ the greatest power of $\overline{\pi}_2$ which divides $\pi_2^8-1$, each of the nodes in the cycle is the root of a reversed binary tree of depth 5.

Finally, the $x$-coordinate of the point $(0,0)$ is $\infty$ and forms a loop. This node is the root of a reversed binary tree of depth 5.

\textbf{Graph $B_{8}$.}
The graph associated with the action of the map $\theta$ on the points belonging to $B_8$ can be studied relying upon the structure of the quotient ring $R/(\pi_2^8+1) R$ which is isomorphic, after factorization in primes of $\pi_2^8+1$, to
\begin{equation*}
R/\overline{\pi}_2 R \times R / (5-8 \omega) R.
\end{equation*}
Let us define $p = 5-8 \omega$. The period of the $x$-coordinate of a point $P=(0,r)$, where $r \not = 0$ in $R/p R$,  under the action of the map $\theta$ is $l_r$, being $l_r$ the smallest among the  positive integers $k$ such that $[\overline{\pi}_2]^{k} P =  P$ or $-P$. The smallest such integer is $56$. In particular $p \mid (\overline{\pi}_2^{56}+1)$. Since there are 112 such points, then there is just a cycle of length 56 whose nodes represent the $x$-coordinate of the 112 points. Moreover each of these nodes is the root of a reversed binary tree of depth one.

\bibliography{Refs}

\providecommand{\bysame}{\leavevmode\hbox to3em{\hrulefill}\thinspace}
\providecommand{\MR}{\relax\ifhmode\unskip\space\fi MR }
\providecommand{\MRhref}[2]{%
  \href{http://www.ams.org/mathscinet-getitem?mr=#1}{#2}
}
\providecommand{\href}[2]{#2}
\begin{thebibliography}{R{\"u}c87}

\bibitem[Car69]{CZ}
L.~Carlitz, \emph{Kloosterman sums and finite field extensions}, Acta
  Arithmetica \textbf{XVI} (1969), no.~2, 179--193.

\bibitem[Gil81]{gil}
William~J. Gilbert, \emph{Radix representations of quadratic fields}, J. Math.
  Anal. Appl. \textbf{83} (1981), no.~1, 264--274.

\bibitem[Jun93]{jung}
D.~Jungnickel, \emph{Finite fields: Structure and arithmetics},
  Bibliographisches Institut, Mannheim, 1993.

\bibitem[LW90]{LW}
Gilles Lachaud and Jacques Wolfmann, \emph{The weights of the orthogonals of
  the extended quadratic binary goppa codes}, IEEE Transactions on Information
  Theory \textbf{36} (1990), no.~3, 686--.

\bibitem[R{\"u}c87]{Ruck}
Hans-Georg R{\"u}ck, \emph{A note on elliptic curves over finite fields}, Math.
  Comp. \textbf{49} (1987), no.~179, 301--304. \MR{MR890272 (88d:11058)}

\bibitem[Shp01]{shpa}
I.~Shparlinski, \emph{On the multiplicative orders of $\gamma$ and
  $\gamma+\gamma^{-1}$ over finite fields}, Finite Fields and Their
  Applications (2001), no.~7, 327--331.

\bibitem[TV04]{vasiga}
J.~Shallit T.~Vasiga, \emph{On the iteration of certain quadratic maps over
  ${G}{F}(p)$}, Discrete Mathematics (2004), no.~277, 219--240.

\bibitem[Wit01]{wi}
C.~Wittmann, \emph{Group structure of elliptic curves over finite fields},
  Journal of Number Theory \textbf{88} (2001), 335--344.

\end{thebibliography}
\end{document}